\RequirePackage{fix-cm}
\documentclass[smallextended]{svjour3}
\smartqed
\usepackage[colorlinks]{hyperref}
\usepackage{anysize}
\usepackage{color}
\usepackage{graphicx}
\usepackage{graphics}
\usepackage{makeidx}
\usepackage{showidx}
\usepackage{latexsym}
\usepackage{amssymb}
\usepackage{verbatim}
\usepackage{amsmath}
\usepackage{amsfonts}
\newtheorem{Theorem}{Theorem}
\newtheorem{Proposition}{Proposition}

\newtheorem{cor}{Corollary}

\begin{document}
\title{Warped Product Conformal Kahler Manifolds and Kenmotsu Structures}
\titlerunning{Warped Product Conformal Kahler Manifolds}
\author{Hassan Attarchi}
\authorrunning{H. Attarchi}
\institute{H. Attarchi \at
              Department of Mathematics and Computer Science,\\
              Amirkabir University of Technology, Tehran, Iran. \\
              Tel.: +98-911-175-1603\\
              \email{hassan.attarchi@aut.ac.ir}}
\date{Received: date / Accepted: date}
\maketitle
\begin{abstract}
In this paper, the conditions on the warped product manifold of a Kenmotsu manifold and the real line are considered such that the warped manifold becomes a conformal Kahler manifold. This result demonstrates the close relation of Kenmotsu and Kahler manifolds.
\end{abstract}
\section{Introduction}
To study manifolds with negative curvature, Bishop and O'Neill introduced the notion of warped product as a generalization of Riemannian product~\cite{onil}. Afterwards, warped product was used in physical issues for modeling the standard space-time, especially for the neighborhoods of stars and black holes~\cite{onil2}. In 1960's and 1970's, when almost contact manifolds were studied as an odd dimensional counterpart of almost complex manifolds, the warped product was used to make examples of almost contact manifolds~\cite{blair72,tanno69}. There are different classifications of almost contact structures which one of the most significant classes is trans-sasakian manifolds. The trans-sasakian manifolds are divided into three subgroups named: Sasakian, Kenmotsu and Cosymplectic manifolds~\cite{munteanu}. In this classification, Kenmotsu manifolds are generated locally by warped product of a Kahler manifold and an interval of the real line $I\!\!R$~\cite{kenmotsu}.\par
Recently, G. Ganchev and V. Mihova in~\cite{ga&mi}, used warped product to generate Kahler manifolds of $\alpha$-Sasakian manifolds. First, they proved that the warped product manifold gets a complex structure. Then, they used a conformal change on the metric of warped manifold to make a Kahler one. In our work, a Kenmotsu manifold is used as underlying manifold of a warped product manifold and the possible structures of this warped product manifold are investigated.\par
Aiming at our purpose, we organize this paper as follows. In Section 2, the structure of an almost contact metric manifold $(M,\varphi,\eta,\xi,g)$ is studied and the definition of trans-sasakian, $\alpha$-Sasakian, $\beta$-Kenmotsu manifolds and some useful equations related to them are reviewed. In Section 3, we find out the conditions on warped product manifold of a Kenmotsu manifold and the real line to be a conformal Kahler manifold. Conversely, the conformal Kahler manifold $\bar{M}$ is supposed to be warped product of odd dimensional manifold $M$ and the real line $I\!\!R$. Then, the conditions which the odd dimensional manifold $M$ is an almost Kenmotsu manifold are investigated. These results demonstrate the close relation of Kenmotsu and Kahler manifolds.
\section{Preliminaries and Notations}
An \emph{almost contact metric structure} $(\varphi, \xi, \eta, g)$ on a (2m+1)-dimensional
manifold $M$ is defined by $\varphi\in\otimes_{1}^{1}(M)$, $\xi\in\chi(M)$
and $\eta\in\Lambda^{1}(M)$ satisfy the following properties
$$\varphi^2X=-X+\eta(X)\xi,\ \ \eta(\xi)=1,\ \ \eta(\varphi)=0,\ \ \varphi\xi=0,\ \ g(\varphi X,\varphi Y)=g(X,Y)-\eta(X)\eta(Y),$$
for all $X,Y\in\Gamma TM$~\cite{blair}. An almost contact metric manifold $M$ is called a \emph{trans-Sasakian}
manifold if it satisfies
\begin{equation}~\label{11}
(\nabla_{X}\varphi)Y=\alpha(g(X,Y)\xi-\eta(Y)X)+\beta(g(\varphi X,Y)\xi-\eta(Y)\varphi X),
\end{equation}
where $\nabla$ is Levi-Civita connection of Riemannian metric $g$ and $\alpha,\beta$ are
smooth functions on $M$. The almost contact metric manifold $M$ which satisfies~(\ref{11}) is called trans-Sasakian manifold of type $(\alpha,\beta)$. Note that trans-Sasakian manifolds of type $(\alpha,0)$ are $\alpha$-\emph{Sasakian} manifolds, trans-Sasakian manifolds of type $(0,\beta)$ are $\beta$-\emph{Kenmotsu} manifolds and trans-sasakian
manifolds of type $(0,0)$ are \emph{cosymplectic} manifolds~\cite{munteanu}. Let $M$ be a $\beta$-Kenmotsu manifold, Therefore, $M$ satisfies the followings
\begin{equation}~\label{1}
\begin{array}{l}
\ \ \ \ \ \ \ \ \ \ \ \ \ \ \ \ (\nabla_{X}\varphi)Y=\beta(g(\varphi X,Y)\xi-\eta(Y)\varphi X),\cr
\nabla_{X}\xi=\beta(X-\eta(X)\xi),\ \ \ \ \ (\nabla_{X}\eta)Y=\beta(g(X,Y)-\eta(X)\eta(Y)),
\end{array}
\end{equation}
for all $X,Y\in\Gamma TM$. Moreover, it is known that $\eta$ is closed (i.e. d$\eta$=0) on a $\beta$-Kenmotsu manifold. \par
The \emph{Kahler form} $\Omega_{g}$ is defined on an almost contact metric manifold as follows
$$\Omega_{g}(X,Y):=g(X,\varphi Y),$$
for all $X,Y\in\Gamma TM$. Let $(M,\varphi,\eta,\xi,g)$ be a $\beta$-Kenmotsu manifold, then Kahler form $\Omega_{g}$ satisfies the equation $d\Omega_{g}=\beta(\eta\wedge\Omega_{g})$~\cite{olzak}. Moreover, An almost contact manifold which holds the equations $d\Omega_{g}=\beta(\eta\wedge\Omega_{g})$ and $d\eta=0$ is called \emph{almost Kenmotsu} manifold in literatures~\cite{dileo}.\par
On the other hand, a Hermitian manifold $M$ with Kahler form $\Omega$ which satisfies $d\Omega=\omega\wedge\Omega$ for 1-form $\omega$ is called Locally Conformal Kahler (Conformal Kahler) manifold if $\omega$ is a closed (exact) form~\cite{13}.
\section{Warped product of Kenmotsu manifolds}
Let $(M,\varphi,\eta_{0},\xi_{0},g)$ be a $\beta_{0}$-Kenmotsu manifold and $I\!\!R$ be the real line with natural coordinate system $\bar{\xi}:=\frac{\partial}{\partial t}$ and $\bar{\eta}:=dt$. Consider the warped product manifold $\bar{M}=\mathbb{R}\times_{p}M$ with warp function $p:\mathbb{R}\longrightarrow\mathbb{R}$ and metric $G$ defined as follows
\begin{equation}~\label{met}
G:=\bar{\eta}\otimes\bar{\eta}+p^2(t)g.
\end{equation}
For all $t\in\mathbb{R}$, the manifold $(M_{t}=\{t\}\times M,\varphi,p(t)\eta_{0},\frac{\xi_{0}}{p(t)},p^2(t)g)$ is $\frac{\beta_{0}}{p(t)}$-Kenmotsu manifold, as well. To simplify equations on $\bar{M}$, we use the following notations
\begin{equation}~\label{k s 1}
\beta:=\frac{\beta_{0}}{p},\ \ \ \eta:=p\eta_{0},\ \ \ \xi:=\frac{\xi_{0}}{p}.
\end{equation}
The almost complex structure $J$ on $\bar{M}$ is defined as follows
\begin{equation}~\label{WPK9}
J|_{D}:=\varphi,\ \ \ J\xi:=-\bar{\xi},\ \ \ J\bar{\xi}:=\xi,
\end{equation}
where $D=\{X\in T\bar{M}|\ \eta(X)=\bar{\eta}(X)=0\}$ and it is called \emph{contact distribution} of Kenmotsu manifold $M$. From which any vector field $\bar{X}\in\Gamma T\bar{M}$ can be uniquely written as follows
\begin{equation}~\label{tajzie}
\bar{X}=X+\bar{\eta}(\bar{X})\bar{\xi}=X_D+\eta(X)\xi+\bar{\eta}(\bar{X})\bar{\xi},
\end{equation}
where $X_D\in\Gamma D$ and $X\in\Gamma TM$. Moreover, according to~(\ref{WPK9}) we obtain
\begin{equation}~\label{WPK8}
J\bar{X}=\varphi(X_D)+\bar{\eta}(\bar{X})\xi-\eta(X)\bar{\xi}.
\end{equation}
The equations~(\ref{k s 1})--(\ref{WPK8}) show that $(\bar{M},G,J)$ is an almost Hermitian manifold. In the following, we prove that the almost complex structure $J$ of almost Hermitian manifold $\bar{M}$ is integrable (i.e. its Neijenhuis tensor $N_J$ is vanished). Therefore, $(\bar{M},G,J)$ becomes a Hermitian manifold. Let $\nabla$ be the Levi-Civita connection on $(\bar{M},G,J)$, then
$$(\nabla_{\bar{X}}J)\bar{Y}=(\nabla_{X}J)Y+(\nabla_{X}J)(\bar{\eta}(\bar{Y})\bar{\xi})
+\bar{\eta}(\bar{X})(\nabla_{\bar{\xi}}J)Y+\bar{\eta}(\bar{X})(\nabla_{\bar{\xi}}J)
(\bar{\eta}(\bar{Y})\bar{\xi}),$$
for all $\bar{X},\bar{Y}\in\Gamma T\bar{M}$. By using some properties of Kenmotsu manifolds presented in Section 2, the components of above equation are calculated as follows
\begin{equation}~\label{WPK7}
\left\{
\begin{array}{l}
(\nabla_{X}J)Y=\beta g(\varphi X,Y)\xi-\beta\eta(Y)J(X)-\beta g(X,Y)\bar{\xi}-\eta(Y)(\bar{\xi}\ln p)X,\cr
(\nabla_{X}J)(\bar{\eta}(\bar{Y})\bar{\xi})=\beta\bar{\eta}(\bar{Y})X-\beta\bar{\eta}(\bar{Y})\eta(X)\xi-\bar{\eta}(\bar{Y})(\bar{\xi}\ln p)JX,\cr
\bar{\eta}(\bar{X})(\nabla_{\bar{\xi}}J)Y=\bar{\eta}(\bar{X})\eta(Y)(\bar{\xi}\ln p)\bar{\xi},\cr
\bar{\eta}(\bar{X})(\nabla_{\bar{\xi}}J)(\bar{\eta}(\bar{Y})\bar{\xi})=\bar{\eta}(\bar{X})\bar{\eta}(\bar{Y})(\bar{\xi}\ln p)\xi.
\end{array}
\right.
\end{equation}
The Neijenhuis tensor $N_J$ of almost complex structure $J$ can be written in terms of Levi-Civita connection as follows
\begin{equation}~\label{WpK11}
N_J(\bar{X},\bar{Y})=(\nabla_{J\bar{X}}J)\bar{Y}+J(\nabla_{\bar{Y}}J)\bar{X}
-J(\nabla_{\bar{X}}J)\bar{Y}-(\nabla_{J\bar{Y}}J)\bar{X}.
\end{equation}
Then, by inserting~(\ref{WPK8}) and (\ref{WPK7}) in~(\ref{WpK11}), we obtain that $N_J(\bar{X},\bar{Y})=0$,
Therefore, the almost complex structure $J$ on $\bar{M}$ is a complex one. The following proposition shows how far the 2-form $\Omega_G(.,.):=G(.,J.)$ on $\bar{M}$ is to Kahler structure.
\begin{Proposition}~\label{d omega}
Let $(\bar{M},G,J)$ be the warped product manifold introduced in this Section. then, the Kahler form $\Omega_{G}$ of Riemannian metric $G$ on $\bar{M}$ satisfies following equation
$$d\Omega_{G}=(2d(\ln p)+\beta_{0}\eta_{0})\wedge\Omega_{G}.$$
\end{Proposition}
\proof A straightforward calculation shows\par
\centerline{$\Omega_{G}(\bar{X},\bar{Y})=G(\bar{X},J\bar{Y})=p^2g(X,\varphi Y+\bar{\eta}(\bar{Y})\xi)-\bar{\eta}(\bar{X})\eta(Y),$}
\centerline{$=p^2g(X,\varphi Y)+\eta(X)\bar{\eta}(\bar{Y})-\bar{\eta}(\bar{X})\eta(Y)=p^2g(X,\varphi Y)+(\eta\wedge\bar{\eta})(\bar{X},\bar{Y}).$}
Therefore,
\begin{equation}~\label{WPK2}
\Omega_{G}=p^2\Omega_{g}+\eta\wedge\bar{\eta},
\end{equation}
by using the facts $d\eta\wedge\bar{\eta}=0$ and $d\bar{\eta}=0$ the following is obtained
$$d\Omega_{G}=dp^2\Omega_{g}+p^2d\Omega_{g}+d\eta\wedge\bar{\eta}-\eta\wedge d\bar{\eta}=p^2(2d(\ln p)+\beta_0\eta_0)\wedge\Omega_{g}.$$
From~(\ref{WPK2}) and $(2d(\ln p)+\beta_0\eta_0)\wedge\eta\wedge\bar{\eta}=0$ the last equation can be written in the following form
\begin{equation}~\label{WPK1}
d\Omega_{G}=(2d(\ln p)+\beta_{0}\eta_{0})\wedge\Omega_{G},
\end{equation}
and this completes the proof.
$\Box$\par
\begin{Theorem}~\label{WPK4}
Let $(\bar{M},G,J)$ be the warped product manifold introduced in this Section. Then, $\bar{M}$ is a conformal Kahler manifold if and only if $\beta\eta$ is exact. Moreover, if $\beta\eta=-2du$ for some $u\in C^\infty(\bar{M})$, then $e^{2(u-\ln p)}G$ will be a Kahler metric on $\bar{M}$.
\end{Theorem}
\proof
From Proposition~\ref{d omega}, it is obvious that $\bar{M}$ is conformal Kahler if and only if $2d(\ln p)+\beta_{0}\eta_{0}$ is exact or equivalent to it $\beta\eta=\beta_{0}\eta_{0}$ is exact. Then, by a straightforward calculations, using Proposition~\ref{d omega} and $-2du=\beta\eta=\beta_{0}\eta_{0}$ the following is obtained\par
\centerline{$d\Omega_{\bar{G}}=d(e^{2(u-\ln p)})\wedge\Omega_{G}+e^{2(u-\ln p)}d\Omega_{G},$}
\centerline{$=e^{2(u-\ln p)}(2du-2d(\ln p))\wedge\Omega_{G}+e^{2(u-\ln p)}(2d(\ln p)+\beta_{0}\eta_{0})\wedge\Omega_{G}=0,$}
and this completes the proof.
$\Box$\par
\begin{cor}~\label{WPK12}
Let $(M,\varphi,\eta_{0},\xi_{0},g)$ be a $\beta_{0}$-Kenmotsu manifold. If $\beta_{0}\eta_{0}$ is exact (closed), then $\bar{M}=\mathbb{R}\times_{p}M$ is a (locally) conformal Kahler manifold. Especially, for the simple connected $1$-Kenmotsu manifold $M$, the warped manifold $\bar{M}=\mathbb{R}\times_{p}M$ is a conformal Kahler manifold.
\end{cor}
By considering the above Theorem and Corollaries, it is proved that there are Kahler and Kenmotsu manifolds of any dimensions. We start with $\mathbb{R}^2$ and its natural Kahler structure. There is a $3$-dimension simply connected Kenmotsu manifold $M=\mathbb{R}\times_{f}\mathbb{R}^2$ where $f(t)=ce^t$~\cite{blair,kenmotsu}. Then $\mathbb{R}\times_{p}M$ is a simply connected conformal Kahler manifold by Corollary~\ref{WPK12} and consequently we have a Kahler manifold of dimension $4$. Continuing the current method, we produce Kahler and Kenmotsu manifolds of any dimensions $n\geq2$.\par
To complete the previous discussion, we go into a converse problem. It means that in the warped product conformal Kahler manifold $\bar{M}=\mathbb{R}\times_{p}M$ with metric $G$, when manifold $M$ is a Kenmotsu one. In~\cite{tashiro}, it is proved that $M$ is an almost contact manifold as an orientable hypersurface of $\bar{M}$. Using the previous notations, we suppose that the almost contact structure of $M$ is given by $(\varphi,\eta_0,\xi_0,g)$. Let $\Omega_G$ be Kahler form of $\bar{M}$, since $\bar{M}$ is considered to be conformal Kahler then there is the smooth function $f:\bar{M}\longrightarrow\mathbb{R}$ such that $f\Omega_G$ is closed. Therefore,
\begin{equation}~\label{d omega G}
0=d(f\Omega_G)=pdf\wedge\eta_0\wedge\bar{\eta}+pfd\eta_0\wedge\bar{\eta}+p^2df\wedge\Omega_g+2fpdp\wedge\Omega_g+p^2fd\Omega_g,
\end{equation}
and the followings are obtained
\begin{cor}~\label{lem}
Let $\bar{M}$ be a conformal Kahler manifold furnished by the above structure. If $df=\xi_0(f)\eta_0$ and $dp=0$, then $M$ is an almost Kenmotsu manifold.
\end{cor}
\proof
Inserting $df=\xi_0(f)\eta_0$ and $dp=0$ in~(\ref{d omega G}), the following is obtained
\begin{equation}
\begin{array}{l}
pfd\eta_0\wedge\bar{\eta}+p^2\xi_0(f)\eta_0\wedge\Omega_g+p^2fd\Omega_g=0,\cr
\Longrightarrow pfd\eta_0\wedge\bar{\eta}=0 \ \& \ p^2(\xi_0(f)\eta_0\wedge\Omega_g+fd\Omega_g)=0,\cr
\Longrightarrow d\eta_0=0 \ \& \ d\Omega_g=-\xi(\ln f)\eta_0\wedge\Omega_g.
\end{array}
\end{equation}
The last equations shows that $M$ is an almost $-\xi(\ln f)$-Kenmotsu manifold and proof is complete.
$\Box$\par
\begin{cor}~\label{lem1}
Let $\bar{M}$ be a conformal Kahler manifold furnished by the above structure. If $df=\frac{\partial f}{\partial t}\bar{\eta}$, then $M$ is a contact manifold.
\end{cor}
\proof
From~(\ref{d omega G}), it is obtained
\begin{equation}
\begin{array}{l}
pfd\eta_0\wedge\bar{\eta}+p^2\frac{\partial f}{\partial t}\bar{\eta}\wedge\Omega_g+2fp\frac{\partial p}{\partial t}\bar{\eta}\wedge\Omega_g+p^2fd\Omega_g=0,\cr
\Longrightarrow d\Omega_g=0 \ \& \ pf\bar{\eta}\wedge d\eta_0=(p^2\frac{\partial f}{\partial t}+2fp\frac{\partial p}{\partial t})\bar{\eta}\wedge\Omega_g,\cr
\Longrightarrow pf=p^2\frac{\partial f}{\partial t}+2fp\frac{\partial p}{\partial t}=\frac{\partial p^2f}{\partial t}\ \ \ \ \& \ \ \ \ \Omega_g=d\eta_0,
\end{array}
\end{equation}
and $\Omega_g=d\eta_0$ shows that $M$ is a contact manifold.
$\Box$\par
\begin{acknowledgements}
The author would like to thank Iran National Science Fundation (INSF), for financial support on this work under grant no. ????????.
\end{acknowledgements}

\end{document}